\documentclass[12pt]{amsart}
\usepackage{tgtermes}

\usepackage[margin=2cm,a4paper]{geometry}
\usepackage{amsmath,amssymb,amsthm}
\usepackage{microtype}
\usepackage[hidelinks]{hyperref}
\usepackage[foot]{amsaddr}
\usepackage{enumitem}
\hypersetup{
  pdftitle={Coloring graphs with no long induced path},
  pdfkeywords={induced paths, chromatic number, clique number, chi-boundedness},
  pdfauthor={Sang-il Oum}
}
\usepackage{keytheorems}
\usepackage{zref-clever}
\zcsetup{cap=true}
\newkeytheorem{theorem}[name=Theorem,refname={Theorem,Theorems}, Refname={Theorem,Theorems}]
\newkeytheorem{lemma,proposition,corollary,conjecture,problem}[sibling=theorem]
\newkeytheorem{remark}[sibling=theorem,style=remark,name=Remark,refname={Remark,Remarks},Refname={Remark,Remarks}]

\begin{document}
\title{Coloring graphs with no long induced path}
\author{Sang-il Oum}
\address{Discrete Mathematics Group, Institute for Basic Science (IBS), Daejeon,~South~Korea.}
\address{Department of Mathematical Sciences, KAIST, Daejeon, South Korea.}
\thanks{Supported by the Institute for Basic Science (IBS-R029-C1).}
\date{September 11, 2026}

\begin{abstract}
Let $P_t$ denote the induced path on $t$ vertices.
Let $\omega(G)$ denote the maximum number of vertices in a clique of a graph $G$.
Gy\'arf\'as (1987) proved that
every $P_t$-free graph $G$ satisfies
$\chi(G)\le(t-1)^{\omega(G)-1}$,
and Gravier, Ho\`ang, and Maffray (2003) improved this to
$\chi(G)\le (t-2)^{\omega(G)-1}$ for $t\ge4$.
We lower the base of the exponential by one:
for every $t\ge5$, every $P_t$-free graph $G$ satisfies
\[
\chi(G)\le 3\,(t-3)^{\omega(G)+4}.
\]
The proof combines two refinements of the Gy\'arf\'as path argument
and was developed with the assistance of Claude Fable 5.1 of Anthropic and GPT Pro of OpenAI.
\end{abstract}
\maketitle

\section{Introduction}
All graphs are finite and simple. We write $P_t$ for the path on $t$
vertices, and a graph is \emph{$P_t$-free} if it has no induced subgraph
isomorphic to $P_t$.
The \emph{chromatic number} $\chi(G)$ is the minimum number of colors needed to color the vertices of $G$ so that adjacent vertices receive distinct colors.
The \emph{clique number} $\omega(G)$ of $G$ is the maximum number of vertices of $G$ pairwise adjacent.
For $S\subseteq V(G)$ we write $\chi(S)=\chi(G[S])$ and
$\omega(S)=\omega(G[S])$, both zero for $S=\varnothing$.
For a vertex $v$ of $G$,
$N(v)$ is the set of all neighbors of~$v$, and
$N[v]=N(v)\cup\{v\}$. Disjoint sets are \emph{complete} to each other if
every possible edge between them is present and \emph{anticomplete} if none
is.
Throughout the paper, $t$ denotes a fixed integer with $t\ge5$, and we write $d=t-3\ge2$.
For a nonempty graph $H$, a set $R\subseteq V(H)$ is \emph{clique-reducing} in $H$ if
$\omega(H[R])<\omega(H)$; equivalently, $V(H)\setminus R$ meets every maximum clique of $H$.

A hereditary class $\mathcal C$ of graphs is \emph{$\chi$-bounded} if
there is a function $f$ such that
for every graph $G\in \mathcal C$, $\chi(G)\le f(\omega(G))$.
Such a function $f$ is called a \emph{$\chi$-binding function} for $\mathcal C$.
If $f$ can be taken as a polynomial, then we say that $\mathcal C$ is \emph{polynomially $\chi$-bounded}.

Gy\'arf\'as~\cite{Gyarfas1987} proved that
$P_t$-free graphs are $\chi$-bounded, with $\chi(G)\le(t-1)^{\omega(G)-1}$,
and Gravier, Ho\`ang, and Maffray~\cite[Corollary~1.4]{GHM2003} improved this to
$(t-2)^{\omega(G)-1}$ for $t\ge4$. Whether $P_t$-free graphs are
\emph{polynomially} $\chi$-bounded is open for every $t\ge6$.
For $t=5$, Scott, Seymour, and Spirkl~\cite{SSS2023} proved the
near-polynomial bound $\chi(G)\le\omega(G)^{\log_2\omega(G)}$ when
$\omega(G)\ge3$, and very recently Nguyen~\cite{Nguyen2025} proved a
polynomial bound. Our result is therefore of interest mainly for $t\ge6$.

We do not address the polynomial
question here. We show instead that the exponential base $t-2$ can be
lowered to $t-3$ for every $t\ge5$, at the cost of a constant factor
depending only on~$t$.

\begin{theorem}\label{thm:main}
  Let $t\ge5$ be an integer and let $d=t-3$.
  Every $P_t$-free graph $G$ satisfies
  $\chi(G)\le \kappa_d\, d^{\,\omega(G)-1}$, where
  \[
    \kappa_d=2d^5+2d^4-3d^3-d^2+2d+4+\frac{6d^2-6d+4}{(d-1)^3}\le 3d^5.
  \]
  In particular, $\chi(G)\le 3(t-3)^{\omega(G)+4}$.
\end{theorem}

We found no general $\chi$-binding function in the literature for
$P_t$-free graphs with exponential base below $t-2$, for any $t\ge6$.\footnote{We note that two surveys
\cite[Theorem~8]{SR2019} and \cite[Section~6]{CK2025} display the bound
$4\cdot3^{\omega-1}$ for $P_6$-free graphs with attribution to~\cite{GHM2003};
but \cite{GHM2003} proves $(\mu-1)^{\omega-1}$, that is $4^{\omega-1}$ for
$P_6$-free graphs.}
The constant $\kappa_d$ is not optimized.
Our bound comes from an explicit recurrence, which we now describe, and
which gives better constants for small $t$.
Define integers~$p_j$ and~$L_m$ by
\begin{equation}\label{eq:pL}
  p_0=0,\quad p_1=1,\quad p_{j+1}=d\,p_j-p_{j-1}\quad(j\ge1),
  \qquad
  L_m=\sum_{j=1}^{m+1}p_j\quad(m\ge0),
\end{equation}
and integers $F_k$, $B_k$ by $F_0=B_0=0$, $F_1=B_1=1$, and, for $k\ge2$,
\begin{align}
  T_k&=\min\Bigl\{B_{k-1},\;
    \min_{0\le m\le k-2}\bigl(B_m+L_mF_{k-m-1}-p_mF_{k-m-2}\bigr)\Bigr\},
    \label{eq:T}\\
  B_k&=(d-1)F_{k-1}+T_k,\label{eq:B}\\
  F_k&=F_{k-1}+B_k.\label{eq:FB}
\end{align}

\begin{theorem}\label{thm:rec}
  Let $t\ge5$ be an integer.
  Every $P_t$-free graph $G$ with $\omega(G)=k\ge1$ has a clique-reducing
  set $R$ with $\chi(G\setminus R)\le B_k$.
  Consequently, $\chi(G)\le F_k$.
\end{theorem}

\zcref{thm:rec} is the graph-theoretic core of the paper; \zcref{thm:main}
follows from it by analyzing the sequence $(F_k)$.

\medskip
\noindent\textbf{Method.}
We describe the classical Gy\'arf\'as path argument and then explain how we
modify it to improve the bound. Let $(c_j)_{j\ge0}$ be an increasing sequence
of prospective bounds. Suppose that $G$ is a connected $P_t$-free graph that
is not $c_k$-colorable, where $k=\omega(G)$, and that every induced subgraph
$H$ of $G$ with $\omega(H)<k$ is $c_{\omega(H)}$-colorable.
The classical Gy\'arf\'as path argument
starts by picking a vertex~$u_1$
and building an induced path from~$u_1$.
Since $\omega(N(u_1))<k$, $G[N(u_1)]$ is $c_{k-1}$-colorable.
Since $G$ is not $c_k$-colorable,
$G\setminus N(u_1)$ has a component $D_1$ that is not
$(c_k-c_{k-1})$-colorable.
We pick a vertex $u_2$ in $N(u_1)$ having a neighbor in $D_1$.
Similarly, for each $i$, suppose that we have an induced path
$u_1u_2\cdots u_{i+1}$ and a connected induced subgraph $D_i$ that contains
a neighbor of $u_{i+1}$, is disjoint from $\bigcup_{j=1}^{i}N(u_j)$,
and is not $(c_k-i c_{k-1})$-colorable. Provided that
$c_k-(i+1)c_{k-1}\ge0$, we can color
$V(D_i)\cap N(u_{i+1})$ with $c_{k-1}$ colors and
pick a component $D_{i+1}$ of $D_{i}\setminus N(u_{i+1})$
and a vertex $u_{i+2}\in N(u_{i+1})\cap V(D_i)$
such that
$u_1u_2\cdots u_{i+2}$ is an induced path,
$D_{i+1}$ is not $(c_k-(i+1)c_{k-1})$-colorable,
and $u_{i+2}$ has a neighbor in $D_{i+1}$ but no other vertex of the path has a neighbor in $D_{i+1}$.

Gy\'arf\'as~\cite{Gyarfas1987} obtained an induced path
$u_1u_2\ldots u_t$ and $D_{t-1}$. This construction can be carried out when
$c_k=(t-1)c_{k-1}$, yielding the bound $(t-1)^{\omega(G)-1}$.

Gravier, Ho\`ang, and Maffray~\cite[Corollaries~1.2 and~1.4]{GHM2003}
obtained their bound through clique-hypergraph coloring: every $P_t$-free
graph has a vertex coloring with $t-2$ colors in which no maximal clique
of size at least two is monochromatic. When $\omega(G)\ge2$, each color
class therefore has clique number less than $\omega(G)$, and induction
gives the bound $(t-2)^{\omega(G)-1}$.
The same bound can be recovered directly from the path construction
above as follows. Since $G$ is $P_t$-free, $D_{t-3}$ has to be complete
to $u_{t-2}$.
Then $D_{t-3}$ is $c_{k-1}$-colorable.
Thus, if $c_k=(t-2)c_{k-1}$, this contradicts the assumption that
$D_{t-3}$ is not $(c_k-(t-3)c_{k-1})$-colorable.
This yields the same bound $(t-2)^{\omega(G)-1}$.

Our argument stops this path one step earlier, at
$u_1u_2\cdots u_{t-3}$ and $D_{t-4}$, so
that only $t-4$ multiples of $c_{k-1}$ are consumed along the path. The
price is that $D_{t-4}$ must then be colored with about $c_{k-1}$ colors,
where the preceding path argument would spend $2c_{k-1}$.
Let $A=N(u_{t-3})\cap V(D_{t-4})$ and let
$\mathcal F$ be the set of components of $D_{t-4}\setminus A$. Since $G$ is
$P_t$-free, a vertex of $A$ with a neighbor in some $F\in\mathcal F$ is
complete to $F$; in particular $\omega(F)<k$ for every $F\in\mathcal F$.
We use two refinements of this situation.

The first refinement is \emph{clique reduction}. Instead of coloring each
$F\in\mathcal F$ completely, we reserve in $F$ a clique-reducing set $R_F$
whose complement is cheap to color. The set $A\cup\bigcup_F R_F$ still has
clique number less than~$k$, so it costs only $c_{k-1}$ colors, and the
complements $F\setminus R_F$ are pairwise anticomplete. Used alone, this
gives the recurrence $a_k=(t-2)a_{k-1}-a_{k-2}$ of \zcref{rem:reduction},
whose base $\tfrac12\bigl(t-2+\sqrt{t(t-4)}\bigr)$ is slightly below $t-2$.

The second refinement is a \emph{large-small split}, appearing 
in the proof of 
Scott, Seymour, and Spirkl~\cite[Lemma~2.1]{SSS2023}.
Fix a threshold $m$
and call $F\in\mathcal F$ \emph{large} if $\omega(F)>m$ and \emph{small}
otherwise. The small components are pairwise anticomplete, so together they
need only $c_m$ colors. The large components, together with the vertices of
$A$ having no neighbor in a large component, induce a graph of clique number
less than~$k$, and so need only $c_{k-1}$ colors. Every part is therefore
cheap except the set $U$ of vertices of $A$ having a neighbor in a large
component. Pick a vertex~$s_F$ in each large component~$F$. A clique of
$N(s_F)\cap U$ is complete to a maximum clique of $F$, which has more than
$m$ vertices, and therefore $\omega(N(s_F)\cap U)<k-m$, whereas
$\omega(G[U])$ is only bounded by $k-1$. In other words, $U$ may have large
clique number, but it is \emph{covered} by vertices that see only small
cliques in it. A rooted form of the Gy\'arf\'as path argument colors such a
covered set $U$ with about $c_{k-m-1}\,d^{\,m}$ colors, so the cost is
exponential in the gap $m$ between the two clique numbers rather than in
$k$. Used alone, this gives the base $t-3$, but with a prefactor that grows
with $\omega(G)$: with $c_j=g_jd^{\,j-1}$ the recurrence becomes roughly
$g_k\le g_{k-1}+g_m/d^{\,k-m}+g_{k-m-1}$, and the last term cannot be made
summable.

The combination is to apply clique reduction \emph{inside} the covered-set
argument. This changes the cost of a covered set from $c_{k-m-1}d^{\,m}$ to
$c_{k-m-1}L_m$, where $L_m$ grows like $\rho_d^{\,m}$ with
$\rho_d=\tfrac12(d+\sqrt{d^2-4})<d$ (and like $\binom{m+2}{2}$ when $d=2$).
Now the loss term $L_m g_{k-m-1}/d^{\,m+1}$ is summable over $k$ with a
balanced choice $m\approx k/2$, and the prefactor stays bounded. This is
\zcref{thm:main}. Keeping track of every saving yields the recurrence
\eqref{eq:T}--\eqref{eq:FB} and \zcref{thm:rec}.

\section{A rooted clique-reduction lemma}\label{sec:prefix}
The following lemma isolates the clique-reduction argument in a form that
remembers the length of the path already built
in the Gy\'arf\'as path argument.

\begin{lemma}\label{lem:prefix}
  Let $G$ be a $P_t$-free graph, let $X\subseteq V(G)$ with
  $\omega(G[X])\le r$, and let $a,b\ge0$ be integers.
  Suppose that every nonempty induced subgraph $H$ of $G[X]$ with
  $\omega(H)<r$
  \begin{itemize}
    \item satisfies $\chi(H)\le a$, and
    \item has a clique-reducing set $R_H$ with $\chi(H\setminus R_H)\le b$.
  \end{itemize}
  Let $p_1p_2\cdots p_\ell$ be an induced path of $G$ with
  $1\le\ell\le t-3$ and $p_\ell\in X$, and let $D$ be a nonempty connected
  induced subgraph of $G[X]$, disjoint from the path, such that $p_\ell$ has
  a neighbor in $D$ and $p_1,\ldots,p_{\ell-1}$ are anticomplete to $D$.
  Then there is $R\subseteq V(D)$ such that
  \[
    \omega(G[R])<r
    \qquad\text{and}\qquad
    \chi(D\setminus R)\le (t-\ell-3)\,a+b.
  \]
\end{lemma}
\begin{proof}
  We use downward induction on $\ell$.
  Suppose first that $\ell=t-3$.
  Let $A=N(p_\ell)\cap V(D)$, and let $\mathcal F$ be the set of components
  of $D\setminus A$. Every vertex $u\in A$ is complete or anticomplete to
  each $F\in\mathcal F$: otherwise, since $F$ is connected, there are
  adjacent $x,y\in V(F)$ with $ux\in E(G)$ and $uy\notin E(G)$, and
  $p_1,\ldots,p_\ell,u,x,y$ induce a $P_t$, because $p_1,\ldots,p_{\ell-1}$
  are anticomplete to $D$ and $x,y\notin N(p_\ell)$.
  Since~$D$ is connected and $A\neq\emptyset$, each $F\in\mathcal F$ has a
  neighbor in $A$, which is then complete to $F$; hence $\omega(F)<r$.
  Also $\omega(G[A])<r$, because $A\cup\{p_\ell\}\subseteq X$ and $p_\ell$ is
  complete to $A$.

  For each $F\in\mathcal F$, choose a clique-reducing set $R_F$ of $F$ with
  $\chi(F\setminus R_F)\le b$, and let
  \[
    R=A\cup\bigcup_{F\in\mathcal F}R_F.
  \]
  We claim that $\omega(G[R])<r$. Let $Q$ be a clique of $G[R]$.
  If $Q\subseteq A$, then $|Q|<r$.
  Otherwise, because distinct components in $\mathcal F$ are anticomplete,
  $Q$ meets $R_F$ for exactly one $F\in\mathcal F$ and $Q\subseteq A\cup R_F$.
  Every vertex of $Q\cap A$ has a neighbor in $F$ and is therefore complete
  to~$F$. If $K$ is a maximum clique of $F$, then $(Q\cap A)\cup K$ is a
  clique of $G[X]$, so it has at most $r$ vertices, and
  $|Q\cap R_F|\le\omega(F[R_F])<\omega(F)=|K|$.
  Hence $|Q|<|(Q\cap A)\cup K|\le r$, proving the claim.
  The graph $D\setminus R$ is the disjoint union of the graphs
  $F\setminus R_F$ for $F\in\mathcal F$, which are pairwise anticomplete,
  so $\chi(D\setminus R)\le b=(t-\ell-3)a+b$.

  Now suppose that $\ell<t-3$, and let $R=N(p_\ell)\cap V(D)$.
  Then $\omega(G[R])<r$ because $R\cup\{p_\ell\}\subseteq X$.
  Suppose that $\chi(D\setminus R)>(t-\ell-3)a+b$, and let $E$ be a
  component of $D\setminus R$ with $\chi(E)>(t-\ell-3)a+b$.
  Since $D$ is connected and $R\neq\emptyset$, some $u\in R$ has a neighbor
  in $E$. Then $p_1\cdots p_\ell u$ is an induced path with $u\in X$, and
  $E$ is a nonempty connected induced subgraph of $G[X]$ disjoint from this
  path such that $u$ has a neighbor in $E$ and $p_1,\ldots,p_\ell$ are
  anticomplete to $E$. By the induction hypothesis for $\ell+1$, there is
  $R_E\subseteq V(E)$ with $\omega(G[R_E])<r$ and
  $\chi(E\setminus R_E)\le(t-\ell-4)a+b$.
  Since $\chi(G[R_E])\le a$, we get $\chi(E)\le(t-\ell-3)a+b$, a
  contradiction.
\end{proof}

\section{The covered-set lemma}\label{sec:covered}
We first record the elementary properties of the sequences in
\eqref{eq:pL}.
\begin{lemma}\label{lem:pL}
  For all $m\ge2$ we have $L_m=dL_{m-1}-L_{m-2}+1$, with $L_0=1$ and
  $L_1=d+1$.
  Moreover $0\le p_m\le L_m\le\sum_{i=0}^{m}d^{\,i}$ for all $m\ge0$, and
  \[
    \sum_{m\ge0}L_mz^m=\frac{1}{(1-z)(1-dz+z^2)}
    \qquad\text{for } |z|<\rho_d^{-1},
    \quad\text{where } \rho_d=\tfrac12\bigl(d+\sqrt{d^2-4}\bigr).
  \]
  In particular $L_m=\binom{m+2}{2}$ when $d=2$, and
  $L_m=O(\rho_d^{\,m})$ with $\rho_d<d$ when $d\ge3$.
\end{lemma}
\begin{proof}
  Since $dp_j=p_{j+1}+p_{j-1}$ for $j\ge1$, we have
  $dL_{m-1}-L_{m-2}+1
  =\sum_{j=2}^{m+1}p_j+\sum_{j=0}^{m-1}p_j-\sum_{j=1}^{m-1}p_j+p_1=L_m$.
  The sequence $(p_j)$ is nondecreasing and nonnegative, since
  $p_{j+1}-p_j=(d-2)p_j+(p_j-p_{j-1})\ge0$ by induction, so
  $0\le p_m\le L_m$; and $p_{j+1}\le dp_j$ gives $p_j\le d^{\,j-1}$ and
  $L_m\le\sum_{i=0}^m d^{\,i}$.
  The generating function of $(p_j)$ is $z/(1-dz+z^2)$, and
  $\sum_m L_mz^m=\sum_{j\ge1}p_j\sum_{m\ge j-1}z^m
  =\frac{1}{1-z}\sum_{j\ge1}p_jz^{j-1}$.
  The roots of $1-dz+z^2$ are $\rho_d^{\pm1}$, and for $d=2$ the
  generating function is $(1-z)^{-3}$.
\end{proof}

Here is the covered-set lemma.
\begin{lemma}\label{lem:covered}
  Let $G$ be a $P_t$-free graph, let $A$ and $S$ be disjoint subsets of
  $V(G)$, let $q\ge1$ be an integer, and let $b,c$ be integers with
  $0\le b\le c$. Suppose that
  \begin{itemize}
    \item every vertex of $A$ has a neighbor in $S$,
    \item $\omega(N(s)\cap A)\le q$ for every $s\in S$, and
    \item every nonempty induced subgraph $H$ of $G[A]$ with
      $\omega(H)\le q$ satisfies $\chi(H)\le c$ and has a clique-reducing set
      $R_H$ with $\chi(H\setminus R_H)\le b$.
  \end{itemize}
  Define $\alpha_0=c$, $\beta_0=b$ and, for $m\ge1$,
  \begin{equation}\label{eq:alphabeta}
    \beta_m=c+(d-2)\alpha_{m-1}+\beta_{m-1},\qquad
    \alpha_m=\alpha_{m-1}+\beta_m.
  \end{equation}
  Then for every $m\ge0$, every nonempty induced subgraph $H$ of $G[A]$
  with $\omega(H)\le q+m$ satisfies $\chi(H)\le\alpha_m$ and has a
  clique-reducing set $R_H$ with $\chi(H\setminus R_H)\le\beta_m$.
  Moreover,
  \begin{equation}\label{eq:coveredexact}
    \alpha_m=L_mc-p_m\cdot (c-b).
  \end{equation}
  In particular, if $\omega(G[A])\le q+m$, then $\chi(G[A])\le L_mc$.
\end{lemma}
\begin{proof}
  Both sequences in \eqref{eq:alphabeta} are nondecreasing.
  We proceed by induction on $m$; the case $m=0$ is the hypothesis.
  Let $m\ge1$ and let $H$ be a nonempty induced subgraph of $G[A]$ with
  $\omega(H)\le q+m$. If $\omega(H)<q+m$, then the claim follows from the
  induction hypothesis and monotonicity. If $H$ is disconnected, we apply
  the connected case to each component, take the union of the resulting
  clique-reducing sets, and reuse colors. So we may assume that $H$ is
  connected and $\omega(H)=r:=q+m$.
  Choose $s\in S$ having a neighbor in $V(H)$, and let $W=N(s)\cap V(H)$.
  Then $\omega(G[W])\le q$, so $\chi(G[W])\le c$.

  Let $D$ be a component of $H\setminus W$. If $\omega(D)<r$, let $R_D=V(D)$.
  Otherwise, since $H$ is connected and $W\neq\emptyset$, some $v_D\in W$
  has a neighbor in $D$. The vertices $s,v_D$ form an induced path with
  $v_D\in V(H)$, and $s$ is anticomplete to $D$ because $V(D) \cap N(s)=\emptyset$.
  By the induction hypothesis,
  every nonempty induced subgraph of $H$ with clique number less than
  $r=q+m$ 
is $\alpha_{m-1}$-colorable and has a
  clique-reducing set whose complement is $\beta_{m-1}$-colorable.
  Thus \zcref{lem:prefix} applies with $X=V(H)$, $\ell=2$, $a=\alpha_{m-1}$, and
  $b=\beta_{m-1}$, and yields $R_D\subseteq V(D)$ with
  \[
    \omega(G[R_D])<r
    \qquad\text{and}\qquad
    \chi(D\setminus R_D)\le(d-2)\alpha_{m-1}+\beta_{m-1}.
  \]

  Let $R$ be the union of the sets $R_D$ over all components $D$ of
  $H\setminus W$. Distinct components are anticomplete, so
  $\omega(G[R])<r=\omega(H)$, and $R$ is clique-reducing in $H$.
  Coloring $W$ with $c$ colors and the pairwise anticomplete graphs
  $D\setminus R_D$ with a common set of colors, we obtain
  \[
    \chi(H\setminus R)\le c+(d-2)\alpha_{m-1}+\beta_{m-1}=\beta_m.
  \]
  Since $\omega(G[R])\le q+m-1$, the induction hypothesis gives
  $\chi(G[R])\le\alpha_{m-1}$, and therefore
  $\chi(H)\le\alpha_{m-1}+\beta_m=\alpha_m$.

  Finally, $\alpha_1=dc+b$, and eliminating $\beta$ from
  \eqref{eq:alphabeta} gives $\alpha_m=d\alpha_{m-1}-\alpha_{m-2}+c$ for
  $m\ge2$. The right-hand side of \eqref{eq:coveredexact} has the same
  initial values and, by \zcref{lem:pL}, satisfies the same recurrence, which
  proves \eqref{eq:coveredexact}. For the last assertion take $b=c$, using
  the empty set as clique-reducing set in each nonempty graph of clique
  number at most~$q$.
\end{proof}

The essential feature is that $W$ is paid for in the complement of~$R$,
not retained in $R$. After deleting~$W$, different components can use
different vertices $v_D$, and their clique-reducing sets can be united
without creating a new maximum clique.
Without clique reduction, the argument would remove $d$ neighborhoods of
cost $\alpha_{m-1}$ each and give $\alpha_m\le c+d\alpha_{m-1}$, that is, a
covered set would cost about $d^{\,m}c$ instead of~$L_mc$.

\section{Proof of the main recurrence}\label{sec:rec}
\begin{proof}[Proof of \zcref{thm:rec}]
  We first check that the sequences are well behaved.
  By \zcref{lem:pL}, $L_m\ge p_m\ge0$, and by induction $F_{k-m-1}\ge F_{k-m-2}$,
  so every candidate in \eqref{eq:T} is nonnegative and $T_k\ge0$.
  Hence $F_k\ge dF_{k-1}$ and $B_k\ge(d-1)F_{k-1}\ge F_{k-1}\ge B_{k-1}$,
  while $B_k=F_k-F_{k-1}\le F_k$. Thus $(F_k)$ and $(B_k)$ are nondecreasing
  and $B_k\le F_k$ for all $k$.

  The second assertion follows from the first by induction on $k$: if $R$
  is a clique-reducing set of $G$ with $\chi(G\setminus R)\le B_k$, then
  $\omega(G[R])\le k-1$, so $\chi(G)\le\chi(G[R])+\chi(G\setminus R)
  \le F_{k-1}+B_k=F_k$.

  We prove the first assertion by induction on $k$; by what we just said,
  the induction hypothesis includes the bound $\chi(H)\le F_{\omega(H)}$
  for every induced subgraph $H$ of $G$ with $\omega(H)<k$.
  If $k=1$, then $G$ is edgeless and $R=\emptyset$ works, since
  $\chi(G)=1=B_1$.
  Let $k\ge2$. If $G$ is disconnected, we apply the connected case to each
  component with clique number $k$ and take $R$ to be the union of the
  resulting sets together with all components of smaller clique number.
  So we may assume that $G$ is connected.

  Let $u_1$ be a vertex of $G$ and let $R=N(u_1)$; then $\omega(G[R])<k$,
  and it remains to show that $\chi(G\setminus R)\le B_k$.

  Suppose that $\chi(G\setminus N(u_1))>B_k=(d-1)F_{k-1}+T_k$.
  By the induction hypothesis, every induced subgraph $H$ of $G$ with
  $\omega(H)<k$ satisfies $\chi(H)\le F_{k-1}$, and in particular
  $\chi(G[N(u)])\le F_{k-1}$ for every vertex $u$.
  By the Gy\'arf\'as path argument applied inside $G\setminus N(u_1)$, we
  obtain an induced path $u_1u_2\cdots u_{t-3}$
  and a connected induced subgraph~$D$ of $G\setminus N(u_1)$ such that
  \begin{itemize}
    \item $u_{t-3}$ has a neighbor in $D$,
    \item $D$ is disjoint from $\bigcup_{i=1}^{t-4} N(u_i)$, and
    \item $\chi(D)>B_k-(t-5)F_{k-1}=F_{k-1}+T_k$.
  \end{itemize}
  Let $A=V(D)\cap N(u_{t-3})$ and let $\mathcal F$ be the set of components
  of $D\setminus A$. Since $G$ is $P_t$-free, if a vertex~$a$ in $A$ has a
  neighbor in $F\in \mathcal F$, then $a$ is complete to $F$: otherwise $F$
  contains an edge $xy$ such that $a$ is adjacent to $x$ and non-adjacent to
  $y$, and $u_1,u_2,\ldots, u_{t-3},a,x,y$ induce a $P_t$.
  Since $D$ is connected and $A\neq\emptyset$, every $F\in\mathcal F$ has a
  vertex having a neighbor in $A$. Thus $\omega(F)\le k-1$ for every
  $F\in\mathcal F$, and $\omega(G[A])\le k-1$ because $A\subseteq N(u_{t-3})$.
  We now bound $\chi(D)$ in two ways.

  \emph{Clique reduction.}
  By \zcref{lem:prefix} applied with $X=V(G)$, $r=k$, $\ell=t-3$, the path
  $u_1\cdots u_{t-3}$, $a=F_{k-1}$ and $b=B_{k-1}$, whose hypotheses hold by
  induction, there is $R_D\subseteq V(D)$ with $\omega(G[R_D])<k$ and
  $\chi(D\setminus R_D)\le B_{k-1}$. Hence
  $\chi(D)\le F_{k-1}+B_{k-1}$.

  \emph{Large-small split.}
  Let $0\le m\le k-2$ and $q=k-m-1\ge1$.
  We say $F\in \mathcal F$ is \emph{large} if $\omega(F)>m$ and \emph{small}
  otherwise. Let $U$ be the set of vertices of $A$ having a neighbor in a
  large component. For each large component $F$ pick a vertex $s_F$ of $F$,
  and let $S$ be the set of these vertices. Then $S\subseteq V(D)\setminus A$
  is disjoint from $U$, and every vertex of $U$ is complete to some large
  component~$F$ and hence adjacent to~$s_F$.
  For each large $F$, every vertex of $N(s_F)\cap U$ is complete to~$F$, so
  each clique $Q$ of $G[N(s_F)\cap U]$ together with a maximum clique of~$F$
  forms a clique of $G$; hence $|Q|\le k-\omega(F)\le q$, that is,
  $\omega(N(s_F)\cap U)\le q$.
  Every nonempty induced subgraph $H$ of $G[U]$ with $\omega(H)\le q$
  satisfies, by induction, $\chi(H)\le F_q$ and has a clique-reducing set
  whose complement is $B_q$-colorable. Since $\omega(G[U])\le k-1=q+m$ and
  $F_q-B_q=F_{q-1}$, \zcref{lem:covered} with $c=F_q$ and $b=B_q$ gives
  \[
    \chi(G[U])\le L_mF_q-p_mF_{q-1}.
  \]
  In every small component $F$, choose by induction a clique-reducing set
  $R_F$ with $\chi(F\setminus R_F)\le B_m$; when $m=0$ there are no small
  components and $B_0=0$. Let
  \[
    Z=(A\setminus U)\cup\bigcup_{F\text{ large}}V(F)\cup\bigcup_{F\text{ small}}R_F.
  \]
  We claim that $\omega(G[Z])\le k-1$. The vertices of $A\setminus U$ are
  anticomplete to the large components, and distinct components are
  anticomplete. So a clique $Q$ of $G[Z]$ lies in $A\setminus U$, in a
  large component, or in $(A\setminus U)\cup R_F$ for a small component $F$.
  In the first two cases $|Q|\le k-1$. In the third case, as in the proof of
  \zcref{lem:prefix}, replacing $Q\cap R_F$ by a maximum clique of $F$ gives
  a clique of $G$ with more than $|Q|$ vertices, so $|Q|\le k-1$.
  Thus $\chi(G[Z])\le F_{k-1}$. The graphs $F\setminus R_F$ for small $F$
  are pairwise anticomplete, so together they are $B_m$-colorable.
  Therefore
  \[
    \chi(D)\le\chi(G[Z])+B_m+\chi(G[U])
    \le F_{k-1}+B_m+L_mF_q-p_mF_{q-1}.
  \]

  Taking the better of the two bounds over all choices of $m$, we obtain
  $\chi(D)\le F_{k-1}+T_k$, contradicting the choice of $D$.
\end{proof}

\section{Analysis of the recurrence}\label{sec:analysis}
Throughout this section, $(F_k)$ and $(B_k)$ are the sequences defined by
\eqref{eq:T}--\eqref{eq:FB}; recall from the proof of \zcref{thm:rec}
that they are nondecreasing with $B_k\le F_k$, and that
\begin{equation}\label{eq:Fd}
  F_k=dF_{k-1}+T_k\qquad(k\ge2).
\end{equation}

\subsection{The clique-reduction bound}
\label{rem:reduction}
What happens if we take only the first candidate in \eqref{eq:T}?
From $T_k\le B_{k-1}$, we deduce that 
$B_k\le(d-1)F_{k-1}+B_{k-1}$ and $F_k\le dF_{k-1}+B_{k-1}$.
The sequence defined by \[ 
a_0=0,\quad a_1=1,\quad 
\text{and} \quad 
a_k=(t-2)a_{k-1}-a_{k-2}
\] 
satisfies these relations with equality, because
$a_k-a_{k-1}=(d-1)a_{k-1}+(a_{k-1}-a_{k-2})$; so $F_k\le a_k$ and
$B_k\le a_k-a_{k-1}$ for all $k$. This is the bound obtained from clique
reduction alone, without the large-small split.
The characteristic roots of the recurrence are $\lambda_t^{\pm1}$, where
$\lambda_t=\tfrac12\bigl(t-2+\sqrt{t(t-4)}\bigr)<t-2$, so
$a_k=(\lambda_t^{k}-\lambda_t^{-k})/(\lambda_t-\lambda_t^{-1})$ and, for
$k\ge2$,
\[
  F_k\le a_k\le(t-2)a_{k-1}<(t-2)\frac{\lambda_t^{\,k-1}}{\sqrt{t(t-4)}}
  =\sqrt{1+\textstyle\frac{4}{t(t-4)}}\;\lambda_t^{\,k-1}.
\]

\subsection{A quadratic bound}
Since $B_m\le F_m$ and $p_mF_{q-1}\ge0$, the second candidate in
\eqref{eq:T} gives, for every $0\le m\le k-2$ with $q=k-m-1$,
\begin{equation}\label{eq:basicrec}
  F_k\le dF_{k-1}+F_m+L_mF_q\le dF_{k-1}+F_m+F_q\sum_{i=0}^{m}d^{\,i}.
\end{equation}

\begin{lemma}\label{lem:quadratic}
  For every $k\ge1$, $F_k\le\binom{k+1}{2}d^{\,k-1}$.
\end{lemma}
\begin{proof}
  Let $G_0=0$ and $G_j=\binom{j+1}{2}d^{\,j-1}$ for $j\ge1$, a nondecreasing
  sequence with $G_1=1=F_1$.
  We prove $F_k\le G_k$ by induction on $k$; let $k\ge2$ and
  suppose that $F_j\le G_j$ for $j<k$.
  Let $q=\lfloor\sqrt{k}\rfloor$ and $m=k-1-q$.
  Since $k\ge2$, we have $1\le q\le k-1$, so $0\le m\le k-2$, and
  \eqref{eq:basicrec} gives
  \[
    \frac{F_k}{d^{\,k-1}}
    \le\frac{G_{k-1}}{d^{\,k-2}}+\frac{G_m}{d^{\,k-1}}
    +\frac{G_q}{d^{\,k-1}}\sum_{i=0}^m d^{\,i}.
  \]
  First, $G_{k-1}/d^{\,k-2}=\binom{k}{2}$.

  Second, since $q>\sqrt k-1$, we have $k<(q+1)^2$ and therefore
  $m=k-1-q<(q+1)^2-1-q=q^2+q\le 2^{q+1}$,
  where the last inequality holds for all integers $q\ge1$
  because it holds for $q\in\{1,2\}$
  and $\frac{(x+1)^2+(x+1)}{x^2+x}=\frac{x+2}{x}\le 2$ for all $x\ge2$.
  As $d\ge2$, we deduce that
  \[
    \frac{G_m}{d^{\,k-1}}
    =\frac{m(m+1)}{2d^{\,q+1}}
    \le\frac{m+1}{2}.
  \]

  Third, since $d-1\ge1$ and $q+m=k-1$, we have
  \[
    \frac{G_q}{d^{\,k-1}}\sum_{i=0}^m d^{\,i}
    =\binom{q+1}{2}
    \frac{1}{d^{m+1}}
    \sum_{i=0}^m d^{\,i}
    < 
    \frac{q(q+1)}{2}.
  \]

  Summing up, and using $q^2\le k$, we obtain
  \[
    \frac{F_k}{d^{\,k-1}}
    <\binom{k}{2}+\frac{m+1+q(q+1)}{2}
    =\binom{k}{2}+\frac{k+q^2}{2}
    \le\binom{k+1}{2}.\qedhere
  \]
\end{proof}

\subsection{Combining two bounds}
\begin{proof}[Proof of \zcref{thm:main}]
  By \zcref{thm:rec} it suffices to show $F_k\le \kappa_dd^{\,k-1}$ for all
  $k\ge1$. Let $h_k=F_k/d^{\,k-1}$ for $k\ge1$; thus $h_1=1$.
  For $k\ge2$, choose
  \[
    m=\Bigl\lfloor\frac{k-1}{2}\Bigr\rfloor,\qquad
    q=\Bigl\lceil\frac{k-1}{2}\Bigr\rceil,
  \]
  so that $0\le m\le k-2$ and $q=k-m-1\ge1$.
  Dividing \eqref{eq:basicrec} by $d^{\,k-1}$ and applying
  \zcref{lem:quadratic} to $F_m$ and $F_q$, we obtain
  \begin{equation}\label{eq:increment}
    h_k\le h_{k-1}
    +\frac{\binom{m+1}{2}}{d^{\,q+1}}
    +\frac{L_m\binom{q+1}{2}}{d^{\,m+1}}.
  \end{equation}
  For $k=2j+1$ and $k=2j+2$ we have $m=j$, while $q=j$ and $q=j+1$
  respectively; the middle term of \eqref{eq:increment} vanishes when $m=0$.
  Summing \eqref{eq:increment} over $2\le k'\le k$ therefore gives
  \begin{align*}
    h_k&\le 1+\sum_{j\ge1}\binom{j+1}{2}\Bigl(\frac1{d^{\,j+1}}+\frac1{d^{\,j+2}}\Bigr)
    +\sum_{j\ge0}\frac{L_j\bigl(\binom{j+1}{2}+\binom{j+2}{2}\bigr)}{d^{\,j+1}}
    \\
    &= 
    1+ \frac12 \sum_{j\ge1}\left(\frac{(j+1)j}{d^{j+2}}
    -\frac{j(j-1)}{d^{j+1}}\right)
    + \sum_{j\ge1}\frac{j^2}{d^{j+1}}
    + \sum_{j\ge0} \frac{(j+1)^2 L_j}{d^{j+1}}
    \\ 
    &=
    1+\frac{d+1}{(d-1)^3}
    +\sum_{j\ge0}\frac{(j+1)^2L_j}{d^{\,j+1}},
  \end{align*}
  where we used $\binom{j+1}{2}+\binom{j+2}{2}=(j+1)^2$, the fact that the
  first sum in the second line telescopes to zero, and
  $\sum_{j\ge1}j^2/d^{\,j}=d(d+1)/(d-1)^3$.
  Both series converge: the first equals $(d+1)/(d-1)^3$, and the second
  converges because, by \zcref{lem:pL}, the generating function
  $H_d(z)=\sum_mL_mz^m$ has radius of convergence $\rho_d^{-1}>d^{-1}$
  (for $d=2$, $H_2(z)=(1-z)^{-3}$ has radius $1$).
  Moreover
  \[
    \sum_{j\ge0}\frac{(j+1)^2}{d^{\,j+1}}L_j
    = 
    \sum_{j\ge0}\frac{1+3j+j(j-1)}{d^{\,j+1}}L_j
    =\frac1d\Bigl(H_d(z)+3zH_d'(z)+z^2H_d''(z)\Bigr)\Big|_{z=1/d}.
  \]
  Evaluating this rational function 
  from \zcref{lem:pL} gives
  \[
    h_k\le 1+\frac{d+1}{(d-1)^3}+\sum_{j\ge0}\frac{(j+1)^2L_j}{d^{\,j+1}}
    =2d^5+2d^4-3d^3-d^2+2d+4+\frac{6d^2-6d+4}{(d-1)^3}=\kappa_d.
  \]
  Finally, $(6d^2-6d+4)/(d-1)^3\le16$ for $d\ge2$, and
  $2d+20\le3d^3+d^2$, so $\kappa_d\le2d^5+2d^4\le 3d^5$.
\end{proof}

\section*{Statement of AI use}
Anthropic's Claude Fable 5.1 produced the clique-reduction argument
presented in \zcref{lem:prefix} and \zcref{rem:reduction} after it read
notes generated, at the author's direction, by OpenAI's Codex Astra over
several failed attempts.
OpenAI's GPT Pro produced the large-small split argument presented here
when the author provided an earlier version of this paper containing the
clique-reduction argument. The large-small decomposition has the earlier
$P_5$-free antecedent in Scott, Seymour, and
Spirkl~\cite[Lemma~2.1]{SSS2023} discussed in the introduction.
GPT Pro produced a proof combining the two arguments.
The author simplified the proofs and rewrote the arguments.
The author takes full responsibility for the mathematical correctness.

\section*{Acknowledgements}
The author would like to thank Tung H. Nguyen for his useful remarks on earlier versions of the manuscript.

\bibliographystyle{amsplain}
\providecommand{\bysame}{\leavevmode\hbox to3em{\hrulefill}\thinspace}
\providecommand{\MR}{\relax\ifhmode\unskip\space\fi MR }
\providecommand{\MRhref}[2]{\href{http://www.ams.org/mathscinet-getitem?mr=#1}{#2}
}
\providecommand{\href}[2]{#2}

\end{document}